\documentclass[11pt,a4paper]{article}
\usepackage[utf8]{inputenc}
\usepackage[english]{babel}
\usepackage{a4wide}

\usepackage{amsmath, amssymb}
\usepackage{amsthm}
\usepackage{cleveref}

\title{Zero-dimensional metrizable CDH space \(X\) such that \(X^2\) is not CDH}

\author{Michal Hevessy\footnote{https://orcid.org/0009-0003-4192-2407}
\footnote{The work on this paper was supported by the Czech Science Foundation
grant 24-10705S}\\
Department of Mathematical Analysis\\
Faculty of Mathematics and Physics, Charles University\\
Prague, Czechia\\
E-mail: hevessy@karlin.mff.cuni.cz}
\date{}

\theoremstyle{plain}
\newtheorem{theorem}[]{Theorem}

\newtheorem{question}[]{Question}

\newtheorem*{Definition*}{Definition}
\newtheorem*{definition*}{Definition}

\newtheorem{lemma}[theorem]{Lemma}

\newtheorem*{theorem*}{Theorem}
\newtheorem*{Theorem*}{Theorem}

\newtheorem{corollary}[theorem]{Corollary}

\theoremstyle{definition}
\newtheorem{Definition}[]{Definition}
\newtheorem{definition}[Definition]{Definition}

\newtheorem*{Remarks*}{Remark}
\newtheorem*{remarks*}{Remark}

\newtheorem*{Example*}{Example}
\newtheorem*{example*}{Example}
\newtheorem*{acknowledgement}{Acknowledgement}

\newcommand{\R}{\mathbb{R}}
\newcommand{\N}{\mathbb{N}}

\newcommand{\Q}{\mathbb{Q}}

\newcommand{\closure}[1]{\mkern 1.5mu\overline{\mkern-1.5mu#1\mkern-1.5mu}\mkern 1.5mu}

\begin{document}

\maketitle

\begin{abstract}
        In this paper a construction of a metrizable zero-dimensional CDH space \(X\) such that \(X^2\) has exactly \(\mathfrak{c}\) countable dense subsets is provided. Furthermore, it is shown that the space can be constructed consistently co-analytic. Thus answering an open question asked by Medini. To do so we use the notion of \(\lambda\)-sets.
\end{abstract}

\renewcommand{\thefootnote}{}



\section{Introduction}
    As is common in the literature for countable dense homogeneous spaces by a space we will always mean a separable metrizable topological space. Given a space \(X\) we will denote by \(\mathcal{H}(X)\) the space of homeomorphisms of \(X\). A space \(X\) is called countable dense homogeneous (CDH) if for every two \(A, B\) countable dense subsets of \(X\) there exists \(h \in \mathcal{H}(X)\) such that \(h(A) = B\). Some basic positive results about CDH spaces include that \(2^\omega\), \(\R^n\) for \(n \in \N\), \(\omega^\omega\) or \([0,1]^\omega\) are all CDH spaces. In fact the CDH property is very well understood for \(0\)-dimensional Polish spaces. Recall that a space is Polish if it is completely metrizable and \(0\)-dimensional if it has a base consisting of clopen sets. Let us denote \[\mathcal{C} = \{X; X \approx \kappa \oplus (\lambda \times 2^\omega) \oplus (\mu \times \omega^\omega) \text{, where } 0 \leq \kappa, \lambda, \mu \leq \omega\}.\] We have the following result due to Hrušák and Avillés.
    \begin{theorem}[\cite{Countable_dense_homogeneity_of_definable_spaces}]
    \label{CDH 0-dim space}
        Let \(X\) be a Polish \(0\)-dimensional CDH space. Then \(X \in \mathcal{C}\).
    \end{theorem}
    \begin{remarks*}
        This result in fact holds even for Borel \(0\)-dimensional CDH spaces \cite[Corollary 2.5]{Countable_dense_homogeneity_of_definable_spaces} and consistently for projective sets \cite[Corollary 2.7]{Countable_dense_homogeneity_of_definable_spaces}.
    \end{remarks*}
    Note that the class \(\mathcal{C}\) is closed under finite and even countable products. Thus the following questions asked by Medini \cite{Medini} and later repeated in a survey paper of Hrušák and van Mill \cite{Hrušák_van_Mill} are very natural.
    \begin{question}\cite[Problem 14]{Hrušák_van_Mill}
    \label{question_1}
        Is there a \(0\)-dimensional CDH space \(X\) such that \(X^2\) is not CDH?
    \end{question}
    \begin{question}\cite[Problem 28]{Hrušák_van_Mill}
    \label{question_2}
        For which cardinals \(\kappa\) is there a \(0\)-dimensional CDH space \(X\) such that \(X^2\) has exactly \(\kappa\) many types of countable dense subsets?
    \end{question}
    Or the following.
    \begin{question}\cite[Question 1.9]{Medini}
    \label{question_3}
    Can the space from \Cref{question_1} be consistently analytic or consistently co-analytic?
    \end{question}
Note that due to the remark after \Cref{CDH 0-dim space} we cannot get in ZFC the space answering \Cref{question_1} analytic or co-analytic.

We will construct a \(0\)-dimensional CDH space X, such that \(X^2\) has exactly \(\mathfrak{c}\) many types of countable dense subsets, which is consistently co-analytic. Thus answering \Cref{question_1} in the negative and partially answering \Cref{question_2} and \Cref{question_3}. In \cite{Medini} it has been shown that a space answering \Cref{question_1} exists under MA(\(\sigma\)-centered). However our construction works in ZFC and uses more elementary tools.

\section{Preliminary results}
In the construction we will use the notion of \(\lambda\)-sets. 
\begin{definition}
     A set \(X \subset 2^\omega\) is called a \(\lambda\)-set if every countable subset of \(X\) is \(G_\delta\) in \(X\), i.e. for any \(D \subset X\) countable \(D = \bigcap_{n=1}^{\infty}\left(G_n \bigcap X\right)\), where \(G_n\) are open in \(2^\omega\).
\end{definition}
\begin{remarks*}
    Note that being a \(\lambda\)-set is a hereditary property and that the space \(2^{\omega}\) is not a \(\lambda\)-set. Thus, no \(\lambda\)-set contains a copy of \(2^\omega\).
\end{remarks*}
By \Cref{CDH 0-dim space}, Polish zero-dimensional CDH-spaces behave very predictably. If we want any non-standard behavior we need to focus on non Polish zero-dimensional spaces. One such class of spaces are meager spaces and if we restrict our attention to those, it has been shown in \cite{Countable_dense_homogeneity_and_the_Baire_property} that the notion of a \(\lambda\)-set arises very naturally in the study of CDH-spaces.

In \cite{Countable_dense_homogeneity_and_lambda-sets} it has been shown that there are uncountable \(\lambda\)-sets that are CDH.
\begin{definition}
    \[\mathfrak{b} = \min(\{|F|; F \subset \omega^\omega \text{ and } \forall g \in \omega^\omega \, \exists f \in F \text{ such that } \{x \in \omega; g(x) \leq f(x)\} \text{ is infinite}\}).\]
\end{definition}
\begin{theorem}
[\cite{Countable_dense_homogeneity_and_lambda-sets}]
\label{lambda sets and CDH}
For any cardinal number \(\kappa \leq \mathfrak{b}\) there exists \(\lambda\)-set \(X\) such that \(X\) is CDH.
\end{theorem}
\begin{remarks*}
    Note that we have \( \aleph_1 \leq \mathfrak{b} \leq \mathfrak{c}\). More details regarding \(\mathfrak{b}\) can be found in \cite{Set_theory}.
\end{remarks*}

As we we want to show that the second power of the constructed space is not CDH we will need some tools for finding many distinct countable dense sets.

The following has been used in \cite{Nearly_countable_dense_homogeneous_spaces} for studying spaces with more types of countable dense subsets. We will use it to show that the space we construct has many types of countable dense subsets. 
\begin{lemma}
[\cite{Homogeneity_and_generalization_of_2-point_sets}]
\label{Q has many nowhere dense subsets}
    There exist \(\mathfrak{c}\) many nonhomeomorphic nowhere dense subsets of \(\Q\).
\end{lemma}

In \cite{Countable_dense_homogeneity_and_the_double_arrow_space} the relation of products and CDH spaces has been studied. The following can be seen as an obstruction for product spaces to be CDH.
\begin{theorem}
[\cite{Countable_dense_homogeneity_and_the_double_arrow_space}]
\label{CDH spaces and copies of cantor}
    Let \(X, Y\) be two spaces. If \(X \times Y\) is CDH, then \(X\) contains a subspace homeomorphic to \(2^{\omega}\) if and only if \(Y\) contains a subspace homeomorphic to \(2^{\omega}\).
\end{theorem}
However, we will need a finer tool as we not only want to show that the second power is not CDH but also that it has exactly \(\mathfrak{c}\) many types of countable dense sets. We will actually use the following which follows from the proof of \Cref{CDH spaces and copies of cantor} in \cite{Countable_dense_homogeneity_and_the_double_arrow_space}.
\begin{corollary}
\label{special subsets of products}
    Let \(X, Y\) be spaces such that \(X\) contains a copy of \(2^\omega\) and \(Y\) does not contain a copy of \(2^\omega\). Then there exists countable dense subset \(C\) of \(X \times Y\) such that for no \(D \subset C\) we have \(\closure{D} \approx 2^\omega\).
\end{corollary}

When dealing with the descriptive quality of the constructed space, the following will be needed.
\begin{theorem}
[\cite{Descriptive_set_theory_and_forcing}]
\label{small spaces are consistently co-analytic}
    Suppose MA \(+\) \(\neg\)CH \(+\) \(\omega_1 = (\omega_1)^{\text{L}}\). Then every \(A \subset 2^\omega\) of cardinality \(\aleph_1\) is \(\boldsymbol{\Pi}_1^{1}\).
\end{theorem}
More details regarding the set theoretic axioms can be found in \cite{Set_theory}.

\section{Main results}
We have all the tools needed and can proceed to the construction.
\begin{theorem}
    There exists a \(0\)-dimensional CDH space \(X\) with the following properties:
    \begin{itemize}
        \item \(X^2\) has exactly \(\mathfrak{c}\) many types of countable dense subsets.
        \item If MA + \(\neg\)CH + \(\omega_1 = (\omega_1)^{\text{L}}\) holds, then \(X\) is \(\boldsymbol{\Pi}_1^{1}\).
    \end{itemize}
\end{theorem}
\begin{proof}
    By \Cref{lambda sets and CDH} we can find \(Y\), a \(\lambda\)-set of cardinality \(\aleph_1\) such that \(Y\) is CDH. Consider \(X = Y \oplus 2^\omega\). This space is clearly CDH, since both \(Y\) and \(2^\omega\) are CDH. By \Cref{small spaces are consistently co-analytic}, if MA + \(\neg\)CH + \(\omega_1 = \omega_1^{\text{L}}\) holds then \(Y\) is \(\boldsymbol{\Pi}_1^{1}\) thus also \(X\) is \(\boldsymbol{\Pi}_1^{1}\). Now, note that \(X^2 \approx Y^2 \oplus 2^\omega \oplus 2^\omega \times Y\). First we will show that for any \(h \in \mathcal{H}(X^2)\) we have \(h(2^\omega \times Y) =  2^\omega \times Y\).

    To this end let \(h \in \mathcal{H}(X^2)\). Since any clopen subset of \(2^\omega \times Y\) contains the Cantor set and \(Y^2\) does not contain the Cantor set, we have that \(h(2^\omega \times Y) \bigcap Y^2 = \emptyset\). By the same argument, we have \(h(2^\omega) \bigcap Y^2 = \emptyset\). Now suppose \(h(2^\omega)\bigcap (2^\omega \times Y) \neq \emptyset\). Denote by \(\pi_2: 2^\omega \times Y \to Y\) the projection to the second coordinate. Then we have that \(\pi_2(h(2^\omega)\bigcap (2^\omega \times Y))\) is compact, zero-dimensional metric space. It is also a clopen subset of \(Y\) therefore it is a crowded space. But this means that \(\pi_2(h(2^\omega)\bigcap 2^\omega \times Y) \approx 2^\omega\) which cannot be since \(Y\) does not contain the Cantor set. 
    This and the fact that \(h(2^\omega) \bigcap Y^2 = \emptyset\) implies that \(h(2^\omega) = 2^\omega\) and thus \(h(2^\omega \times Y) = 2^\omega \times Y\).
    Now we will show that \(Y \times 2^\omega\) has exactly \(\mathfrak{c}\) many types of countable dense sets, this together with the fact that it is preserved by any \(h \in \mathcal{H}(X^2)\) will yield the desired.

    Since \(|2^\omega \times Y| = \mathfrak{c}\) we have that the space \(2^\omega \times  Y\) has at most \(\mathfrak{c}\) many types of countable dense subsets. Thus, it is enough to find \(\mathfrak{c}\) many countable dense subsets of \(2^\omega \times Y\) of a different type.
    
    Let \(s \in 2^\omega\) and let \(A, B \subset 2^\omega\) be open disjoint such that \(A \cup B \cup \{s\} = 2^\omega \) and \(\partial A = \partial B = s\). Note that we have \(A \approx B \approx 2^\omega \setminus \{0\}\).

    By \Cref{special subsets of products}, we can find \(Q_0 \subset A \times Y\) countable dense such that for no \(E \subset Q_0\) we have \(\closure{E} \approx 2^\omega\). Let \(\{U_n\}_{n\in \omega}\) be a countable base of the space \(B \times Y\). For \(n \in \omega\) let \(F_n \subset U_n\) be countable such that \(\closure{F_n} \approx 2^\omega\). Let \(Q_1 = \bigcup_{n \in \omega}F_n\) and put \(D = Q_0 \cup Q_1\). Then the set \(D\) is by construction a countable dense subset of \(2^\omega \times Y\). By \Cref{Q has many nowhere dense subsets}, there is a collection \(\{C_r; r \in (0,1)\}\) of countable pairwise nonhomeomorphic nowhere dense subsets of \( \{s\} \times Y\). For \(r \in (0,1)\) let \(D_r = D \cup C_r \). 
    
    Now let \(p,r \in (0,1)\), \(p \neq r\) and suppose there is \(h \in \mathcal{H}(2^\omega \times Y)\) such that \(h(D_p) = D_r\). Let \(a \in  \{s\} \times Y \). First suppose \(h(a) \in A \times Y\), then we can find an open neighborhood \(V_0\) of \(a\) such that \(h(V_0) \subset A \times Y\). We have that \(V_0 \cap (B \times Y) \neq \emptyset\). There exists \(n \in \omega\) such that \(U_n \subset V_0\) thus also \(F_n \subset V_0\) which means \(h(F_n) \subset Q_0\). This contradicts the fact that for no countable subset \(E \subset Q_0\) we have \(\closure{E} \approx 2^\omega\). Now suppose \(h(a) \in B \times Y\). Then there is an open neighborhood \(V_1\) of \(a\) such that \(h(V_1) \subset B \times Y\), thus also \(h(V_1 \cap (A \times Y)) \subset B \times Y\). Again, we can find \(n \in \omega\) and \(F_n \subset h(V_1 \cap (A \times Y))\). However, this means \(h^{-1}(F_n) \subset Q_0\), which cannot be. This means that \
    \(h(\{s\} \times Y) = \{s\} \times Y\), thus \(h(C_p) = C_r\), which is a contradiction. Thus all the sets \(D_r\) for \(r \in (0,1)\) are of a different type.
\end{proof}

\begin{acknowledgement}
The author is very grateful to Benjamin Vejnar for his helpful comments and insights.
\end{acknowledgement}

\bibliographystyle{abbrv}
\bibliography{bibliography}

\include{bibliography.tex}

\end{document}